\documentclass[11pt,twoside]{amsart}
\usepackage{amsxtra}
\usepackage{color}
\usepackage{amsopn}
\usepackage{amsmath,amsthm,amssymb}
\usepackage{mathrsfs,mathtools}
\usepackage{enumerate}
\usepackage{xcolor}

\newtheorem{theorem}{Theorem}[section]

\newtheorem{proposition}[theorem]{Proposition}
\newtheorem{lemma}[theorem]{Lemma}
\newtheorem*{theorem*}{Theorem}
\theoremstyle{definition}

\newtheorem{remark}[theorem]{Remark}

\newcommand{\Exterior}{\mathchoice{{\textstyle\bigwedge}}%
    {{\bigwedge}}%
    {{\textstyle\wedge}}%
    {{\scriptstyle\wedge}}}
\newcommand{\R}{{\mathbb R}}
\renewcommand{\H}{\mathrm{H}}
\newcommand{\C}{{\mathbb C}}

\newcommand{\bR}{\mathbb{R}}

\newcommand{\bH}{\mathbb{H}}

\newcommand{\bO}{\mathbb{O}}

\newcommand{\beq}{\begin{equation}}
\newcommand{\eeq}{\end{equation}}

\newcommand{\f}{\varphi}
\newcommand{\g}{\gamma}

\renewcommand{\l}{\lambda}

\renewcommand{\o}{\omega}

\newcommand{\s}{\sigma}

\newcommand{\U}{{\mathrm U}}
\newcommand{\SU}{{\mathrm{SU}}}

\newcommand{\SO}{{\mathrm {SO}}}
\newcommand{\Sp}{{\mathrm {Sp}}}

\newcommand{\G}{{\mathrm G}}
\newcommand{\K}{{\mathrm K}}

\newcommand{\Sg}{{\mathrm S}}

\newcommand{\M}{{\mathrm M}}

\newcommand{\Spin}{\mathrm{Spin}}

\newcommand{\W}{\wedge}

\DeclareMathOperator\Ad{Ad}
\DeclareMathOperator\ad{ad}

\newcommand{\ga}{\mathfrak{a}}
\newcommand{\gb}{\mathfrak{b}}

\newcommand{\gf}{\mathfrak{f}}
\renewcommand{\gg}{\mathfrak{g}}
\newcommand{\gh}{\mathfrak{h}}
\newcommand{\gk}{\mathfrak{k}}
\newcommand{\gl}{\mathfrak{l}}
\newcommand{\gm}{\mathfrak{m}}

\newcommand{\gp}{\mathfrak{p}}

\newcommand{\gs}{\mathfrak{s}}

\newcommand{\gu}{\mathfrak{u}}

\newcommand{\so}{\mathfrak{so}}
\newcommand{\su}{\mathfrak{su}}

\newcommand{\gsp}{\mathfrak{sp}}

\numberwithin{equation}{section}

\title[Nearly parallel G$_2$-structures with large symmetry group]{Nearly parallel G$_{\mathbf2}$-structures with large symmetry group}

\author{Fabio Podest\`a}
\address{ Dipartimento di Matematica e Informatica "Ulisse Dini", Universit\`a di Firenze, V.le Morgagni 67/A, 50100 Firenze, Italy}
\email{podesta@unifi.it}

\date{\today}


\begin{document}
	\begin{abstract} We prove the existence of a one-parameter family of nearly parallel $\G_2$-structures on the manifold $\Sg^3\times \mathbb R^4$, which are mutually non isomorphic and invariant under the cohomogeneity one action of the group $\SU(2)^3$. This family connects the two locally homogeneous nearly parallel $\G_2$-structures which are induced by the homogeneous ones on the sphere $\Sg^7$. 
		
	\end{abstract}
	
\maketitle
\section{Introduction}

A nearly parallel $\G_2$-structure (NP-structure for brevity) on a $7$-dimensional manifold $\M$ is given by a positive $3$-form $\f\in \Omega^3(\M)$ such that $d\f = \lambda *_\f \f$ for some (non-zero) $\lambda \in \bR$, where $*_\f$ denotes the Hodge star operator relative to the associated Riemannian metric $g$. The name ``nearly parallel'' comes from the fact that only a 1-dimensional component of $\nabla\f$  is different from zero (see \cite{FG}), where $\nabla$ is the Levi Civita connection of $g$, and these structures are also said to have weak holonomy $\G_2$, where this terminology goes back to Gray (\cite{G}). The Riemannian manifold $(M,g)$ is irreducible  Einstein with scalar curvature given by $\frac {21}8\lambda^2$ and the existence of an NP-structure is equivalent to the existence of a spin structure with a non-zero Killing spinor as well as to the existence of a torsion-free $\Spin(7)$-structure on the cone $C(\M) := \bR^+\times \M$ inducing the cone metric $dr^2+r^2g$ (see \cite{Ba}).
More precisely, an NP-structure on a compact simply connected manifold $\M$ will be called {\it proper} if the cone metric on $C(\M)$ has full holonomy $\mathcal H = \Spin(7)$, or equivalently if the space of Killing spinors is one-dimensional. When the NP-structure is not proper and the metric $g$ has not constant curvature, the holonomy $\mathcal H$ reduces either to $\SU(4)$ or further to $\Sp(2)$, corresponding to the existence of a Sasakian (but not $3$-Sasakian) and a $3$-Sasakian stucture on $\M$ respectively. It is known (see \cite{FKMS}) that any $3$-Sasakian manifold admits a second NP-structure which is proper and the squashed sphere $\Sg^7$ is an example of this situation. \par 
In some sense NP-structures are a seven-dimensional analogue of nearly K\"ahler structures in six dimensions, which are automatically Einstein and admit a Killing spinor. Actually the cone metric on the cone over a six-dimensional strict nearly K\"ahler manifold $\rm N$ has holonomy inside $\G_2$ and moreover for both nearly K\"ahler and NP-structures their canonical metric connections $\overline\nabla$ have $\overline\nabla$-parallel, totally skew-symmetric torsion. It is also known that given a six-dimensional strict nearly K\"ahler manifold N, the cone $C(\rm N)$ endowed with the sine-cone metric has an NP-structure (see e.g.~\cite{BG}).\par 
In order to find possibly new examples, it is very natural to investigate manifolds endowed with special structures, as nearly K\"ahler or NP-structures, whose full automorphism group is ample. The classification of compact homogeneous NP-structures was achieved in \cite{FKMS}, where also many useful results were proved on the full automorphism group, while later in \cite{Bu} the classification of compact homogeneous nearly K\"ahler six-dimensional manifolds was obtained. In \cite{PS1},\cite{PS2} the study of compact six-dimensional nearly K\"ahler manifolds which admit a compact Lie group of automorphisms with generic orbits of codimension one was initiated and more recently Foscolo and Haskins (\cite{FH}) proved the existence of completely new, inhomogeneous nearly K\"ahler structures on the sphere $\Sg^6$ and on $\Sg^3\times \Sg^3$, invariant under the cohomogeneity one action of the group $\SU(2)\times \SU(2)$. As for NP-structures, Cleyton and Swann (\cite{CS}) classified all manifolds which carry such a structure with a simple Lie group of automorphisms acting by cohomogeneity one; in strong contrast to the homogeneous case, they found that the standard sphere $\Sg^7$ and $\bR {\rm{P}}^7$ acted on by the exceptional Lie group $\G_2$ are the only complete examples.  \par 
In this work we investigate the existence of $\G$-invariant NP-structures on the manifold $\M \cong \Sg^3\times\mathbb R^4$, which admits a cohomogeneity one (almost effective) action of the group $\G = \SU(2)^3$. The manifold $\M$ can be realized as the complement $\M = \Sg^7\setminus \Sigma$, where $\Sigma\cong \Sg^3$ is one the two singular orbits for a cohomogeneity one action of $\G$ on $\Sg^7$, and it is special in the sense that it already admits a complete $\G$-invariant metric with full holonomy $\G_2$, namely the well-known example constructed by Bryant and Salamon (\cite{BS}) on the spin bundle over $\Sg^3$. The group $\G$ appears in the list of possible groups with can act by cohomogeneity one preserving a $\G_2$-structure (\cite{CS}) and actually it is (locally) isomorphic to the full isometry group of the Bryant-Salamon metric. Moreover, in view of the results in \cite{FKMS}, the automorphism group of an NP-structure on a compact manifold acts transitively on it whenever its dimension is at least $10$, so that the group $\G$ has the highest possible dimension to allow non-homogeneous examples. Principal $\G$-orbits are diffeomorphic to $\rm Y:=\Sg^3\times \Sg^3$ and the non-trivial isotropy representation of a principal isotropy subgroup allows to easily determine the space of invariant $2$- and $3$-forms on $\M$. A $\G$-invariant NP-structure on $\M_o\cong \mathbb R^+\times \rm Y$ given by a $3$-form $\f$ induces a family of so called nearly half-flat $\G$-invariant $\SU(3)$-structures $(\o,\psi_+,\psi_-)$ on $\rm Y$ (see \cite{FIMU}); the $2$-form $\o$ is forced to lie in a one-dimensional subspace of invariant $2$-forms on $\rm Y$ and when these $\SU(3)$-structures are all nearly K\"ahler structures on $\rm Y$ we obtain the well-known example of the sine-cone over the homogeneous nearly K\"ahler manifold $\rm Y$ (see \cite{BM},\cite{FIMU}).

In our main result Theorem \ref{Main} we prove the existence of a one-parameter family $\mathcal F_a$ ($a\in \mathbb R^+$) of $\G$-invariant NP-structures on $\M$, mutually non isomorphic, connecting the two locally homogeneous NP-structures on $\M$ induced by the 
known homogeneous NP-structures on $\Sg^7$; the parameter $a\in \mathbb R^+$ gives a measure of the size of the singular orbit $\Sg^3$. The problem of understanding which of these structures extends over a $\G$-equivariant compactification $\overline\M$ is unsolved, albeit there is some numerical evidence that no such structure might exist besides the homogeneous ones. In case a global $\G$-invariant NP-structure on $\Sg^7$ should exist, we prove that it would be proper and distinct from any of the Einstein metrics of cohomogeneity one on $\Sg^7$ found by B\"ohm (\cite{Bo}). One might expect to find more invariant NP-structures by reducing the group to $\SU(2)^2\times \rm U(1)$, by analogy with what happens for $\G_2$-holonomy metrics on $\M$ (see the recent results in  \cite{FHN}), or further to $\SU(2)^2$ and this can be the object of further investigations.\par 
The work is structured as follows. In the second section we describe the manifold $\M$ together with the $\G$-action as well as all the $\G$-invariant $\G_2$-structures. In section 3 we write down the equations defining the $\G$-invariant NP-structures.  We continue describing the special solutions to the system \eqref{sys2} given by the sine-cone construction over the nearly K\"ahler homogeneous manifold $\Sg^3\times \Sg^3$ and by the two well-known homogeneous NP-structures on $\Sg^7$. We then analyze the symmetries of the system \eqref{sys2}, providing (Prop.\ref{LocEx}) the existence of a two-dimensional family of mutually non isomorphic and non locally homogeneous NP-structures on an open tubular neighborhood of a $\G$-principal orbit. In the last subsection of section 3, we give sufficient and necessary conditions on the solutions of the system \eqref{sys2} on the regular part so that the corresponding NP-structures extend smoothly to an NP-structure on the whole $\M$. In the last section we prove our main Theorem \ref{Main} and the main properties of a global solution in Prop.\ref{global}. 

\bigskip
\noindent {\bf Notation.} Lie groups and their Lie algebras will be indicated by capital and gothic letters respectively.  Given a Lie group $\rm L$ acting on a manifold $\rm N$, for every $X\in \gl$ we will denote by $\hat X$ the corresponding vector field on $\rm N$ induced by the one-parameter subgroup $\exp(tX)$.
\bigskip

\noindent{\bf Acknowledgements.} The author heartily thanks Luigi Verdiani for valuable conversations and his substantial help with the numerical analysis. He also expresses his gratitude to Alberto Raffero for his support with $\G_2$-structures and relative computations and to Anusha Krishnan for her interest. It is a pleasure to thank Simon Salamon for his interest and fruitful discussions.

\section{Preliminaries}
In this section we first consider the non-compact $7$-dimensional manifold $\M$ together with the action of the group $\G\cong \SU(2)^3$ with generic orbits of codimension one. We will then describe the space of all $\G$-invariant $\G_2$-structures on $\M$. \medskip 

\subsection{ The manifold $\M$ and the group action of $\G$.}

We start with the standard (almost effective) action of the compact group 
$\U= \Sp(2)\times\Sp(1)$ on $\mathbb H^2$ given by $(A,q)\cdot v = Av\bar q$, where $(A,q)\in \U$ and $v\in \mathbb H^2$. The sphere $\Sg^7\subset \mathbb H^2$ can be written as the quotient space $\U/\K^+$ with  $\K^+ =\{({\mbox{diag}}(q,q'),q)\in \U\}\cong \Sp(1)\times \Sp(1)$ being the isotropy subgroup at the point $e_1=(1,0)\in \mathbb H^2$. \par 

We consider the action of $\G:= \{ ({\mbox{diag}}(q_1,q_2),q_3)\in \U|\ q_1,q_2,q_3\in\Sp(1) \}\cong \Sp(1)^3 $ on $\Sg^7$. The curve $\gamma:t\mapsto(\cos t,\sin t)\in \Sg^7$ is transverse to the $\G$-orbits and we easily see that 
$$\G_{\g(t)} = \Sp(1)_{\mbox{diag}} =:\H \qquad t\in (0,\pi/2),$$
$$\G_{\g(0)} = \K^+,\qquad \G_{\g(\pi/2)} =\{(q,q',q')\in \G\}  =: \K^-.$$ 
It then follows that $\G$ acts on $\Sg^7$ by cohomogeneity one with principal orbits diffeomorphic to $\Sg^3\times \Sg^3$. We also fix an $\Ad(\H)$-invariant decomposition 
$$\gg = \gh \oplus V^+ \oplus V^-,\ \gm:= V^+\oplus V^-$$
where 
$$V^+ := \{(X,-2X,X)|\ X\in \gsp(1)\},\quad V^-:=  \{(-2X,X,X)|\ X\in \gsp(1)\}. $$
Note that $\gk^\pm = \gh \oplus V^\pm$. We fix the standard basis of $\gsp(1)\cong\su(2)$ given by 
$$h:= \left(\begin{smallmatrix}i&0\\0&-i\end{smallmatrix}\right),\
e:= \left(\begin{smallmatrix}0&1\\-1&0\end{smallmatrix}\right),\ 
v:= \left(\begin{smallmatrix}0&i\\i&0\end{smallmatrix}\right) $$ 
with 
$$[h,e] = 2v,\ [h,v] = -2e,\ [e,v] = 2h$$
and we consider the maps $j_{\pm}:\gsp(1)\to V^\pm$ given by $j_+(X)=(X,-2X,X)$ and $j_-(X)= (-2X,X,X)$. We then define a basis of $\gm$ as follows
$$e_2:= j_+(h),\ e_3:= j_+(e),\ e_4 := j_+(v),$$
$$e_5:= j_-(h),\ e_6:= j_-(e),\ e_7 := j_-(v).$$
We consider the manifold $\M:= \G\times_{\K^+}\bH$, where $\K^+$ acts on $\bH$ via its standard representation. Then $\M$ can be identified with $\Sg^7\setminus (\G\cdot\g(\frac\pi 2))$ and it is an $\bR^4$-bundle over the singular orbit $\G\cdot\g(0) = \G/\K^+\cong \Sg^3$, namely it is diffeomorphic to $\Sg^3\times \mathbb R^4$. The regular open subset $\M_o$ of $\M$ is $\G$-equivariantly diffeomorphic to $(0,\frac \pi 2)\times \G/\H$. In the open manifold $\M_o$ we can identify the tangent spaces $T_{\gamma(t)}\M_o  = \mathbb R\gamma'(t) \oplus \widehat \gm$. Along the curve $\gamma$ we have a frame, again denoted by $\{e_1,\ldots,e_7\}$, that is given by $\mathcal B_t:= \{\gamma'(t), \hat e_2|_{\gamma(t)},\ldots, \hat e_7|_{\gamma(t)}\}$ and its dual coframe will be denoted by $\{e^1,\ldots,e^7\}$.
For basic information on cohomogeneity one manifolds we refer e.g. to \cite{AA}, \cite{AB}.

\subsection{Invariant $\G_2$-structures}

We start recalling some basic facts about $\G_2$-structures. Given a $7$-dimensional manifold $\M$ and its frame bundle $L(M)\to M$, a $\G_2$-structure is a reduction of $L(M)$ to a subbundle $P$ with structure group $\G_2\subset \SO(7)$. It is known that $\G_2$-structures are in one to one correspondence with smooth sections of the associated bundle $\Lambda^3_+(M):= L(M)\times_{\rm{GL}(7,\mathbb R)}\Lambda^3_+(\mathbb R^7)\subset \Lambda^3(M)$, where $\Lambda^3_+(\mathbb R^7)\subset \Lambda^3(\mathbb R^7)$ is the open orbit $\rm{GL}(7,\mathbb R)\cdot\f_o$ through a $3$-form $\f_o$ with stabilizer 
$\rm{GL}(7,\mathbb R)_{\f_o} = \G_2$ (see e.g.\cite{FKMS},\cite{Br}).
A smooth section $\f$ of $\Lambda^3_+(M)$ (hence a $\G_2$-structure on $\M$) determines a Riemannian metric $g_\f$ as follows: at each point $p\in M$ we consider the non-degenerate symmetric bilinear map 
\[
b_\f:T_pM\times T_pM\rightarrow\Lambda^7(T_pM^*),\quad (v,w)\mapsto \frac16\,\iota_v \f \W \iota_w \f \W\f
\]
and if $\{v_1,\ldots,v_7\}$ is any basis of $T_pM$ with dual basis $\{v^1,\ldots,v^7\}$ then for $i,j=1,\ldots,7$
$$b_\f(v_i,v_j) = \beta_\f(v_i,v_j)\ v^1\wedge\ldots\wedge v^7$$
for some non-degenerate matrix $B_\f:= (\beta_\f(v_i,v_j))_{i,j=1,\ldots,7}$; the Riemannian metric $g_\f$ is then given by (see e.g. \cite{H}) 
$$g_\f(v_i,v_j) = (\det(B_\f))^{-1/9} \beta_\f(v_i,v_j).$$

In order to investigate $\G$-invariant $\G_2$-structures on $\M=\G\times_{\K^+}\mathbb H$, we start considering invariant $\G_2$-structures on the open dense submanifold $\M_o$.\par

The description of $\G$-invariant $3$-forms on $\M_o$ is reduced to the study of the space of $\H$-invariant $3$-forms $\Lambda^3(V^*)$, where $V := \mathbb R e_1 +\gm\cong T_{\g(t)}M$ ($t\in (0,\frac \pi 2)$). We first note that 
$$\Lambda^3(V^*)^\H \cong \Lambda^2(\gm^*)^\H + \Lambda^3((V^+)^*) + \Lambda^3((V^-)^*) + $$
$$
+ [\Lambda^2((V^+)^*)\otimes (V^-)^*]^\H + [(V^+)^*\otimes \Lambda^2((V^-)^*)]^\H.$$
Using the standard notation $e^{i_1i_2\ldots i_k} = e^{i_1}\wedge\ldots \wedge e^{i_k}$, we immediately see that the space $\Lambda^2(\gm^*)^\H$ is generated by the form $\omega:=e^{25}+e^{36}+e^{47}$ and that the space $\Lambda^3(V^*)^\H$ is generated by the invariant $3$-forms 
$$ e^1\wedge \omega,\ \f_1:=e^{234}, \f_2:=e^{567},\ \f_3:=e^{237}-e^{246}+e^{345},\ 
\f_4:= e^{267}-e^{357}+e^{456}.$$ 
If we denote by $\f$ a $\G$-invariant $3$-form on $\M_o$, its restriction along $\gamma$ can be written as
\beq \label{phi}\begin{split}
\f|_{\gamma(t)} &= f_0\left(e^{125}+e^{136}+e^{147}\right) + f_1\,e^{234}+f_2\,e^{567} +f_3\left(e^{237}-e^{246}+e^{345} \right) + \\
 &+ f_4\left(e^{267}-e^{357}+e^{456} \right),\end{split}\eeq
for suitable $f_i\in\mathcal{C}^\infty((0,\frac \pi 2 ))$. Let us fix the volume form $e^{1234567}$ along $\gamma$, so that we get an identification $\Lambda^7(V^*)\cong \bR$. Then, the matrix $B_\f$ associated with the symmetric bilinear form $b_\f$ with respect to $\mathcal{B}_t$ is given by (here $\mathbb I$ denotes the $3\times 3$-matrix) 
\[ 
 B_\f = f_0 
 \left( 
\begin{array}{ccc}
 -f_0^2&0&0\\
 0&b_1{\mathbb I}&b_3{\mathbb I}\\
 0&b_3{\mathbb I}& b_2{\mathbb I}
\end{array}
\right),
\]
where 
$$b_1 := f_1f_4-f_3^2,\ b_2:= f_2f_3-f_4^2,\ b_3:= \frac 12 (f_1f_2-f_3f_4).$$
The 3-form $\f$ defines a $\G_2$-structure if and only if 
$B_\varphi$ is definite. In such a case, $g_\f = (\det(B_\f)^{-1/9}B_\f$ is positive definite and
\[
\det(B_\f) = \frac{1}{64}\,f_0^9\left( f_1^2f_2^2-6\,f_1f_2f_3f_4+4\,f_1f_4^3 + 4\, f_2f_3^3 -3\,f_3^2f_4^2\right)^3 \neq 0.
\]
In this case, we will suppose that the parameter $t$ is the arc length parameter along the curve $\gamma$ (hence throughout the following the parametr $t$ will vary in some interval $I=(0,T)$), i.e. 
$$g_\f(e_1,e_1)=1,$$
so that $\det(B_\f) = - f_0^{27}$ or equivalently
\begin{equation}\label{g11}
f_0^2 = - \left(\frac{ f_1^2f_2^2-6\,f_1f_2f_3f_4+4\,f_1f_4^3 + 4\, f_2f_3^3 -3\,f_3^2f_4^2}{4}\right)^{\frac13}.
\end{equation}
This implies that $g_\f$ can be expressed as a block matrix 
\[
 g_\f = 
 \left( 
\begin{array}{ccc}
 1&0&0\\
 0&g_1{\mathbb I}&g_3{\mathbb I}\\
 0&g_3{\mathbb I}& g_2{\mathbb I}
\end{array}
\right),
 \]
where 
$$g_1 := \frac {f_{{3}}^{2}-f_{{1}}f_{{4}}}{f_0^2},\ 
g_2 = \frac{f_{{4}}^{2}-f_{{3}}f_{{2}}}{f_0^2},\
g_3:= \frac{f_{{3}}f_{{4}}-f_{{1}}f_{{2}}}{2f_0^2}$$
together with the positivity condition, which in view of \eqref{g11} can be written as 
\beq
f_{{3}}^{2}-f_{{1}}f_{{4}} > 0,\ f_4^2 - f_2f_3 > 0.\eeq

We can now compute the expression of the 4-form $*_\f\f$, where $*_\f$ denotes the Hodge operator w.r.t. $g_\f$. We easily obtain
\begin{eqnarray*}
*_\f\f &=& 	A\,e^1\W \left[ \left( f_1^2f_2-3\,f_1f_3f_4+2\,f_3^3 \right)\,e^{234}  - (f_1f_2^2-3f_2f_3f_4+2f_4^3)\,e^{567}\right. \\
	& &  + (f_1f_2f_3-2f_1f_4^2+f_3^2f_4)\,(e^{237}-e^{246} +e^{345})	\\
	& & \left. -(f_1f_2f_4-2f_2f_3^2+f_3f_4^2) (e^{267}-e^{357} + e^{456}) \right] \\
	& &	+\left(\frac{ f_1^2f_2^2-6f_1f_2f_3f_4+4f_1f_4^3+4f_2f_3^3-3 f_3^2f_4^2}{4}\right)^{\frac13}\,\left(e^{2356}+e^{2457}+e^{3467} \right),
\end{eqnarray*}	
where 
\[
A\coloneqq f_0\,2^{\frac13}   \left( f_1^2f_2^2-6\,f_1f_2f_3f_4+4\,f_1f_4^3+4\,f_2f_3^3-3\,f_3^2f_4^2 \right)^{-\frac23}.
\]
Using \eqref{g11}, we see that $A = \frac12f_0^{-3}$. Consequently, the 4-form $*_\f\f$ can be rewritten as follows
\beq\label{*}\begin{split}
*_\f\f	&= \frac{1}{2f_0^3}\,e^1\W [  (f_1^2f_2-3\,f_1f_3f_4+2\,f_3^3) \,e^{234}  - (f_1f_2^2-3f_2f_3f_4+2f_4^3)\,e^{567} \\
	&+ (f_1f_2f_3-2f_1f_4^2+f_3^2f_4)\,(e^{237}-e^{246} +e^{345})	\\
	& -(f_1f_2f_4-2f_2f_3^2+f_3f_4^2) (e^{267}-e^{357} + e^{456}) ] \\
	&-f_0^2\,(e^{2356}+e^{2457}+e^{3467} ). 
\end{split} \eeq

In order to compute $d\f$, we need some preliminary remarks. First of all we note that 
$$\Lambda^4(V^*)^\H = \mathbb R e^1\wedge \Lambda^3(\gm^*)^\H + [\Lambda^2((V^+)^*)\otimes
\Lambda^2((V^-)^*)]^\H,$$
where the last summand is generated by the invariant form $\alpha:= e^{2356} + e^{2457} + e^{3467}$. The next lemma follows by straightforward computations.
\begin{lemma} We have the following commutators for $x,y\in \gsp(1)$
	$$[j_\pm(x),j_\pm(y)]_{\gm} = - j_\pm([x,y]),$$
	$$[j_+(x),j_-(y)]_\gm = j_+([x,y]) + j_-([x,y]).$$
\end{lemma}
Using the standard Koszul's formula for the differential of an invariant form $\psi\in\Lambda^k(\gm)^\H$, namely for $X_0,X_1,\ldots,X_k\in\gm$
$$d\psi(X_0,X_1,\ldots,X_k) = \sum_{i<j}(-1)^{i+j}\psi([X_i,X_j]_{\gm},\ldots, \hat X_i,\ldots, \hat X_j,\ldots, X_k),$$
(here the hat denotes omitted terms) we see that 
$$d\f_1 = d\f_2 = 0,\ d\o = 6(\f_3-\f_4),\ d\f_3= d\f_4 = 6\alpha.$$
Therefore we obtain 
\beq\label{dphi}\begin{split}
	d\f|_{\gamma(t)}	&= f_1'e^{1234} +f_2'e^{1567}+\left(f_3'- 6f_0\right)\left(e^{1237}-e^{1246} +e^{1345}\right) \\
	& + \left(f_4'+6f_0\right)\left(e^{1267}-e^{1357} + e^{1456}\right) + 6\left( f_3+f_4\right)\left(e^{2356}+e^{2457}+e^{3467} \right). 
\end{split} \eeq

\section{Invariant Nearly parallel $\G_2$-structures and their equations}

Recall that a $\G_2$-structure is {\em nearly parallel} if the defining 3-form $\f$ satisfies the equation
\beq\label{nPdef}
d\f=\lambda*_\f\f,\eeq
for some non-zero real constant $\lambda$. In this case, the Riemannian metric $g_\f$ induced by $\f$ is Einstein with scalar curvature 
$\mathrm{Scal}(g_\f)=\frac{21}{8}\lambda^2$.
We now consider a $\G_2$-structure induced by a $\G$-invariant $3$-form $\f$, which can be described as in \eqref{phi}. Then $\f$ defines an NP-structure if and only if $f_0,f_1,f_2,f_3,f_4$ satisfy the following equations
\beq\label{sys1}
\left\{\begin{aligned}
	f_1'		&=	\lambda\,\frac1{f_0^3}\left(f_1\, \frac{f_1f_2-f_3f_4}{2} - \,f_3\left( f_1f_4 - f_3^2 \right)\right),&(1)\\
	f_2' 			&= \lambda\, \frac1{f_0^3} \left( f_4\, (f_2f_3-f_4^2) - f_2\, \frac{f_1f_2-f_3f_4}{2}  \right),&(2)\\
	f_3'			&=  6f_0\, + \lambda\, \frac{1}{2f_0^3}\left( f_1\, ( f_2f_3-f_4^2) - f_4\,(f_1f_4-f_3^2)  \right),&(3)\\
	f_4'		&=\, - 6f_0\, +  \lambda\,\frac{1}{2f_0^3}\left( f_3\, ( f_2f_3-f_4^2) - f_2\,(f_1f_4-f_3^2)  \right),&(4)\\
	f_4 &+ f_3 	= - \frac 16\,\lambda\,f_0^2,&(5)\\
	f_0^6		&=  (f_1f_4-f_3^2)\, ( f_2f_3-f_4^2) - \frac 14 (f_1f_2-f_3f_4)^2 > 0 ,&(6)\\
	0 &> f_1f_4 -f_3^2,\ 0 > f_2f_3-f_4^2.&{}
\end{aligned}\right.\eeq

We use  equation (5) in equation (4) and compare it with equation (3). We then get the expression of $f_0'$ in terms of $f_0,\ldots,f_4$ and the system of equations can be written as follows 

\beq\label{sys2}
\left\{ \begin{aligned}
	f_1'&=	\lambda\,\frac1{f_0^3}\left(f_1\, \frac{f_1f_2-f_3f_4}{2} - \,f_3\left( f_1f_4 - f_3^2 \right)\right),&(1)\\
	f_2' &= \lambda\, \frac1{f_0^3} \left( f_4\, (f_2f_3-f_4^2) - f_2\, \frac{f_1f_2-f_3f_4}{2}  \right),& (2)\\
	f_3'	&=  6f_0\, + \lambda\, \frac{1}{2f_0^3}\left( f_1\, ( f_2f_3-f_4^2) - f_4\,(f_1f_4-f_3^2)  \right), & (3) \\
	f_4'	&=\, - 6f_0\, +  \lambda\,\frac{1}{2f_0^3}\left( f_3\, ( f_2f_3-f_4^2) - f_2\,(f_1f_4-f_3^2)  \right),&(4)\\
	f_0' &=\, - \frac{3}{2f_0^4}\left( (f_1+f_3)(f_2f_3-f_4^2) - (f_2+f_4)(f_1f_4-f_3^2)\right),& (5)\\
	f_4 &+ f_3 + \frac 16\,\lambda\,f_0^2 = 0,&(6)\\
	f_0^6	&-  (f_1f_4-f_3^2)\, ( f_2f_3-f_4^2) + \frac 14 (f_1f_2-f_3f_4)^2 = 0,&(7)\\
	0 &> f_1f_4 -f_3^2,\ 0 > f_2f_3-f_4^2,\ f_0\neq 0.&(8)
\end{aligned}\right.\eeq

The following Lemma can be easily verified using a direct computation.  

\begin{lemma}\label{alg} Equations (6) and (7) in \eqref{sys2} hold for all $t\in I$ if and only if they hold at one point in $I$ and equations (1)-(5) are satisfied for all $t\in I$.\end{lemma}
As an immediate corollary, we note that the algebro-differential system \eqref{sys2} can be reduced to the system of ODE's formed by equations (1)-(5) in \eqref{sys2} coupled with initial conditions at a fixed point $t_o\in I$ satisfying equations (6) and (7) at $t_o$, together with the inequalities (8). We will use this point of view when we will construct families of mutually non-isometric and non locally homogeneous NP-structures in a suitable neighbourhood of homogeneous solutions, which we describe in the following subsection.
\begin{remark}\label{rescaling} Note that under the rescaling $\f\mapsto c\cdot \f$ ($c\neq 0$), we have  $g_{c\f} = c^{2/3}\cdot g_{\f}$ and the constant in \eqref{nPdef} $\lambda\mapsto c^{-1/3}\lambda$. This means that we can always fix $\lambda$ to be any non zero real number. Alternatively, one can consider new functions 
$$\tilde f_0(t) := \lambda^2 f_0(t/\lambda),\ \tilde f_i(t) := \lambda^3 f_i(t/\lambda),\ i=1,\ldots,4 $$
which satisfy the system \eqref{sys2} with $\lambda = 1$. \end{remark}
\begin{remark} It is well known that, given a manifold $X$ endowed with an NP-structure with $3$-form $\f$ and associated Riemannian metric $g$, a hypersurface $f:S\to X$ inherits a so-called {\it nearly half-flat} $\SU(3)$-structure given by a $2$-form $\o$ and a $3$-form $\psi_+$ so that 
$$\o := \imath_\nu\f,\ \psi_+:= -\imath_\nu*\f,\ \psi_-:= J\psi_+ = - f^*\f,$$
where $\nu$ denotes the unit normal to $S$ and $J$ is the almost complex structure induced on $S$ by the $\SU(3)$-structure $(\o,\psi_+)$ (see \cite{FIMU}). This nearly half-flat structure $(\o,\psi_+)$ satisfies the condition $$d\psi_- = - 2\o\wedge \o,$$
when the $3$-form $\f$ satisfies $d\f = 4*\f$ (i.e. $\l=4$). In our situation the $\G$-invariant nearly half-flat structures induced on the principal orbit $\G/\H = \Sg^3\times \Sg^3$ have proportional $2$-forms $\o$, as the isotropy representation of $\H$ forces the space of invariant $2$-forms to be one-dimensional. \par

Viceversa given a smooth family of nearly half-flat structures $(\o(t),\psi_+(t))_{t\in \mathbb R}$ on a $6$-dimensional manifold $S$, the $3$-form $\f:= \o\wedge dt - \psi_-$ on $\mathbb R\times S$ defines an NP-structure (with $\l=4$)  if and only if the following equations are fulfilled (Prop.5.2 in \cite{FIMU})
\beq 
\left\{\begin{aligned}\label{nhf}
      \partial_t\psi_- &= 4\psi_+ - d\o,\\
      d\psi_+{}\ &= -\frac 12 \partial_t(\o\wedge\o).
    \end{aligned}\right.\eeq
In \cite{St} it is proved that starting from a nearly half-flat structure on $S$ it is possible to extend it to a smooth one-parameter family of nearly half-flat structures satisfying \eqref{nhf}, hence obtaining an NP-structure on $I\times S$ for some interval $I\subseteq \mathbb R$ (see also \cite{CLSS}). We will prove a local existence result in Proposition \ref{LocEx}.\end{remark}

\medskip 
\subsection{Particular solutions. }\label{hom}
In this subsection we will describe three special solutions to the system \eqref{sys2}, corresponding to known NP-structures. More precisely they are the sine-cone over the homogeneous nearly K\"ahler manifold $\Sg^3\times \Sg^3$ and the two homogeneous NP-structures on the sphere $\Sg^7$.\par 
(a)\ It is known (see \cite{BM},\cite{BG}) that the sine-cone $C_s(Y) = (0,\pi)\times Y$ over a nearly K\"ahler $6$-dimensional manifold $Y$ carries an $NP$-structure inducing the sine-cone metric $dt^2 + \sin^2t\cdot g_Y$. The homogeneous nearly K\"ahler structure on $\Sg^3\times \Sg^3$ is known to be invariant under the group $\SU(2)^3$  (see \cite{Bu}) and therefore it gives rise to a solution $(f_0,\ldots,f_4)$ of the system \eqref{sys2} for $t\in (0,\pi)$, namely 
$$\lambda = 4,\ f_0(t) = -2\sqrt{3}(\sin t)^2,\ 
f_1(t)=f_2(t)= 8(\sin t)^4,$$
$$f_3(t)=-4\sqrt{3}(\sin t)^3(\cos t + \frac 1{\sqrt{3}}\sin t),\ f_4(t)=-4\sqrt{3}(\sin t)^3(-\cos t + \frac 1{\sqrt{3}}\sin t).$$
The metric $g_Y$ is represented by the block matrix $\left(\begin{smallmatrix}  4\mathbb I&-2\mathbb I\\-2\mathbb I&4\mathbb I\end{smallmatrix}\right)$.\par

(b)\ We consider the standard NP-structure $\mathcal P_1$ on $\Sg^7$, inducing  the standard constant curvature metric. Its full automorphism group is $\rm{Aut}(\mathcal P_1) = \Spin(7)\subset \SO(8)$.  
We consider the octonions $\bO =\{a+be,a,b\in \bH\}$ together with the Cayley form $\Phi\in \Lambda^2\bO$ given by 
$$\Phi(x,y,z,w) = \langle x,\frac 12 [y(\bar z w) - w(\bar zy)]\rangle.$$ 
It is known that the group $\Sp(1)\times\Sp(1)\times \Sp(1)$ acting almost faithfully on $\bO$ by 
$$(q_1,q_2,q_3)\cdot (a+be) = q_1a\bar q_3 + (q_2b\bar q_3)e$$
preserves the form $\Phi$ (see \cite{BH},~p.11) and therefore induces a cohomogeneity one action on the round sphere $\Sg^7$ preserving the standard NP-structure $\f$ given by $\f_x := \imath_x\Phi$, $x\in\Sg^7$. If we consider the curve $\g(t) = \cos t + \sin t\cdot e\in \bO$, we see that the corresponding functions $f_i(t)$ are given by
$$\lambda = 4,\quad f_0(t) = -9\sin t \cos t,$$
$$f_1(t) = 27 \sin^4 t,\ f_2(t) = 27 \cos^4 t,\
f_3(t) = f_4(t) = -27 \sin^2 t \cos^2 t.$$  
The metric can be represented by the block matrix $\left(\begin{smallmatrix} 1&0&0\\ 0&a\mathbb I&0\\0&0&b\mathbb I\end{smallmatrix}\right)$ where 
$a= 9\sin^2 t$,\ $b= 9\cos^2 t.$\par 
(c)\ We now consider the non standard NP-structure $\mathcal P_2$ on the {\it squashed} $\Sg^7$, with full automorphism group given by $\rm{Aut}(\mathcal P_2) = \Sp(2)\cdot \Sp(1)\subset \SO(8)$. 
We refer to the exposition in \cite{AS}, \S 8.2, where the authors describe this homogeneous structure using the presentation of $\Sg^7$ as a normal homogeneous space of the group $\Sp(2)\cdot\Sp(1)$.\par 
Along the geodesic $\gamma(t) = (\cos t,\sin t)\in\Sg^7$ we have the following: 
$$\lambda = \frac{12}{\sqrt 5},\qquad f_0(t) = \frac 9{\sqrt 5} \sin t\cdot \cos t,$$
$$ f_1(t) = \frac{27}{\sqrt 5} (3\sin^4 t\cdot\cos^2 t - \frac 15 \sin^6 t),\quad 
f_2(t)= \frac{27}{\sqrt 5} (3\cos^4 t\cdot\sin^2 t - \frac 15 \cos^6 t),$$ 
$$f_3(t) = \frac{27}{\sqrt 5} \sin^2 t\cdot \cos^2 t \cdot \left( \cos^2 t - \frac {11}5 \sin^2 t\right),\quad f_4(t) = \frac{27}{\sqrt 5} \sin^2 t\cdot \cos^2 t \cdot \left( \sin^2 t - \frac {11}5 \cos^2 t\right).$$
 Note that the sign of the constant $\lambda$ is opposite to that indicated in \cite{AS}. The metric can be represented by the block matrix $\left(\begin{smallmatrix} 1&0&0\\ 0&a\mathbb I&c\mathbb I\\0&c\mathbb I&b\mathbb I\end{smallmatrix}\right)$, where 
$$a:= \frac{36}5 \sin^2 t\cdot \left(\frac 54 - \sin^2 t\right),\ b := \frac{36}5 \cos^2 t\cdot \left(\frac 54 - \cos^2 t\right),\ c:= -\frac{36}5 \sin^2 t\cdot \cos^2 t.$$

\subsection{The functions $f_i$'s and the associated nearly parallel structures}
A solution $(f_0(t),\ldots,f_4(t))$ of the system \eqref{sys2} defines an NP-structure on the open subset $J\times \G/\K$ for some subinterval $J\subseteq I$. In this subsection we study the problem when two such NP-structures are (locally) isomorphic. We start with the following proposition, which follows very closely Prop. 4.1 in \cite{PS2}, and characterizes the local homogeneity of a $\G$-manifold with a $\G$-invariant NP-structure. 

\begin{proposition}\label{loc} Let $X$ be a $7$-dimensional manifold endowed with an NP-structure given by a $3$-form $\f$. Assume that the group $\G\cong \SU(2)^3$ acts on $X$ by automorphisms of the $\G_2$-structure by cohomogeneity one and assume moreover that the NP-structure is locally homogeneous. Then $(X,\f)$ is locally (isometrically) isomorphic to the standard sphere or to the squashed sphere endowed with their respective NP-structures.   \end{proposition}
\begin{proof} We fix $p\in X$ and we consider the Lie algebra $\gs$ of germs of automorphisms of $(X,\f)$ with isotropy subalgebra at $p$ denoted by $\gu$. Then it is known that $\gu$ is reductive and that it embeds into $\gg_2$. Moreover, by local homogeneity, we may suppose that $p$ is $\G$-regular, hence $\gu$ contains $\gu_1:= \gu\cap \gg \cong \su(2)$, the isotropy subalgebra $\gg_p$. Therefore, looking at the list of possible reductive subalgebras of $\gg_2$, we see that $\gu$ can be isomorphic to $\su(2)$, $\su(3)$, $\mathbb R+\su(2)$, $\su(2)+\su(2)$, $\gg_2$. Let $\rm S$ be the simply connected Lie group with Lie algebra $\gs$ and let $\rm U$ be the connected Lie subgroup of $S$ with Lie algebra $\gu$. We claim that $\rm U$ is closed in $\rm S$, whence $X$ is locally isomorphic to a globally homogeneous space (see e.g. \cite{Sp}). Suppose on the contrary that $\rm U$ is not closed in $\rm S$. This implies that $\gu$ is not semisimple (see \cite{Mo}, p.615), hence $\gu\cong \mathbb R + \su(2)$ and therefore $\gu = \gu_1+\mathbb R$. As $\gu$ is reductive, we can write $\gs = \gu + V$, where $V\cong \mathbb R^7$ is $\ad(\gu)$ invariant and $\ad(\gu)|_V\subset \su(3)$. Therefore $\ad^\gs|_{\gu_1} = \rho_1 \oplus \ad \oplus\ 4\mathbb R $, where $\rho_1$ is the standard representation of $\gu_1$ on $\mathbb C^2$. On the other hand $\gu_1\subset \gg\subset\gs$ with ${\rm{codim}}_{\gs}\gg = 2$ and therefore $\ad^\gs|_{\gu_1} = 3\ad \oplus 2\mathbb R$, a contradiction.\par 
A direct inspection of the globally homogeneous (hence compact) manifolds with $\G$-invariant NP-structure (see \cite{FKMS}) and admitting a cohomogeneity one action of $\G$ proves our claim.\end{proof}

Let $\f = \sum_{i=0}^4 f_i\f_i$ and $\tilde \f = \sum_{i=0}^4 \tilde f_i\f_i$ be two $\G$-invariant NP-structures, where the $5$-tuples $(f_0,\ldots,f_4), (\tilde f_0,\ldots,\tilde f_4)$ satisfy the system \eqref{sys2} on some interval $J\subseteq I$. Assume then that these two structures are locally isomorphic, i.e. there are two open subsets $W,\ \widetilde W\subseteq J\times \G/\H$ and a diffeomorphism 
$\psi: W\to \widetilde W$ with $\psi^*\tilde\f = \f$ (and therefore inducing an isometry between the induced metrics $g$ and $\tilde g$ respectively).
We will first suppose that $\psi$ does {\it not} map $\G$-orbits onto $\G$-orbits. As $\G$ acts with cohomogeneity one, this means that both $W$ and $\widetilde W$ are locally homogeneous and therefore, by Prop.\ref{loc}, locally isometric to the standard or squashed sphere. As the full automorphisms of the two homogeneous NP-structures on $\Sg^7$, namely $\Spin(7)$ and $\Sp(2)\cdot \Sp(1)$, contain precisely one copy of $\SU(2)^3$ up to conjugation, we can find a local isometry $\tilde\psi :\widetilde W\to \widetilde W$ preserving $\tilde\f$ and with 
$\tilde\psi_*(\psi_*(\gg)) = \gg$. \par 
Therefore we are reduced with the case where $\psi$ preserves $\G$-orbits. Up to some traslation by an element of $\G$ and up to some reparameterization $t\mapsto t+ c$, we can suppose $\psi(\g_t) = \g_{\pm t}$. If $\psi(\g_t) = \g_{-t}$, we can compose $\psi$ with the transformation $(t,xH)\mapsto (-t,xH)$ reducing to  $\psi(\g_t) = \g_{t}$; the corresponding transformation of the functions $f_i$'s reads 
\beq\label{T1} \tau_o:(f_0(t),f_1(t),f_2(t),f_3(t),f_4(t))\mapsto (-f_0(-t),
f_1(-t),f_2(-t),f_3(-t),f_4(-t)).\eeq
The map $\psi$ induces an automorphism $\psi_*$ of the Lie algebra $\gg$ that preserves the regular isotropy $\gh$, as $\psi$ preserves the curve $\g_t$. If $\psi_*$ is inner, it is the conjugation by an element $n\in {\rm N}_\G(\H)$, $n = \sigma\cdot h$ with $h\in \H$ and $\sigma =(\s_1,\s_2,\s_3)$ where $\s_i = \pm {\rm {Id}}\in \SU(2)$. It then follows that $\Ad(g)$ fixes $\o$ as well as all $\f_1,\ldots,\f_4$, so that the functions $f_0,f_1,\ldots,f_4$ remain unchanged. We now examine the case where $\psi_*$ is outer, namely it permutes the simple factors $\gf_i\cong\su(2)$, $i=1,2,3$ of $\gg$. We now describe how the functions $\f_i$'s transform when $\psi_*$ induces the generators $(12)$ and $(13)$ of the symmetric group $S_3$. \par 
Suppose $\psi_*$ permutes the factors $\gf_i$ leaving $\gf_3$ fixed. Then we easily see that $\omega\mapsto -\omega$, while $\f_1,\f_2$ are exchanged as well as $\f_3,\f_4$. The functions $f_i$'s transform accordingly as follows 
\beq\label{12}
\tau_{(12)}: (f_0,f_1,f_2,f_3,f_4) \mapsto (-f_0,f_2,f_1,f_4,f_3).\eeq
When $\psi_*$ induces the permutation $(13)$ on the factors of $\gg$, we can compute the corresponding transformation of the invariant forms as follows
$$\o\mapsto -\o,\ \f_1\mapsto \f_1-\f_2-\f_3+\f_4,\ \f_2\mapsto -\f_2$$
$$\f_3\mapsto -\f_3-3\f_2+2\f_4,\ \f_4\mapsto \f_4-3\f_2$$
so that the corresponding trasformation of the functions $f_i$'s reads
\beq\label{(13)}
\tau_{(13)}: (f_0,f_1,f_2,f_3,f_4) \mapsto (-f_0,f_1,-f_1-f_2-3(f_3+f_4),-f_1-f_3,f_1+2f_3+f_4).\eeq

Using these, we see that the remaining non-trivial permutations induce the following trasformations

\beq\label{(23)}\begin{split}
\tau_{(23)}:& (f_0,f_1,f_2,f_3,f_4) \mapsto (-f_0,-f_1-f_2-3(f_3+f_4),f_2,f_2+f_3+2f_4,-f_2-f_4).\\
\tau_{(123)}:& (f_0,f_1,f_2,f_3,f_4) \mapsto (f_0,-f_1-f_2-3(f_3+f_4),f_1,f_1+2f_3+f_4,-f_1-f_3).\\
\tau_{(132)}:& (f_0,f_1,f_2,f_3,f_4) \mapsto (f_0,f_2,-f_1-f_2-3(f_3+f_4),-f_2-f_4,f_2+f_3+2f_4).\end{split}\eeq

\subsection{Local Existence} 
We consider the regular ODE system given by equations (1)-(5) (for a fixed $\lambda$) in \eqref{sys2}. Any solution of such a system is a curve in $\mathbb R^5$ lying in the subset 
$$C:= \{(a_0,a_1,\ldots,a_4)\in \mathbb R^5|\ a_0\neq 0, 
0>a_1a_4 -a_3^2,\ 0 > a_2a_3-a_4^2,\ R_1=0,\ R_2=0\}$$
where  
\beq\begin{split} R_1(a_0,\ldots,a_4) &:= a_3+a_4+\frac 16 \lambda a_0^2,\\
R_2(a_0,\ldots,a_4) &:= (a_1a_4-a_3^2)\, ( a_2a_3-a_4^2) - \frac 14 (a_1a_2-a_3a_4)^2 - a_0^6.\end{split}\eeq
We fix $t_o = \frac {\pi}4$ and the initial condition 
$$x_o:=\left(-\frac 92, \frac{27}4,\frac{27}4,-\frac{27}4,-\frac{27}4\right)$$
that corresponds to the initial values at $t_o$ of the homogeneous structure $\mathcal P_1$ on the sphere $\Sg^7$ with constant $\l=4$. In a suitable neighborhood $W$ of $x_o$ the set $C\cap W$ is a $3$-dimensional submanifold, as it can be easily verified. Moreover if 
$\Sigma$ is the group of transformations in $\mathbb R^5$ generated by $\tau_{(12)},\tau_{(13)}$, we can shrink $W$ so that 
$\tau(W)\cap W=\emptyset$ for every $\tau\in \Sigma$. Now if $F(t)$ is the (homogeneous) solution starting from $x_o$, we can fix a $2$-dimensional submanifold $\mathcal S$ in $C\cap W$ that is transversal to the trace of $F$; solutions starting from points in 
$\mathcal S$ are all mutually non equivalent and not locally homogeneous. Therefore we have proved the following 
\begin{proposition}\label{LocEx} There exists a $2$-dimensional family of mutually non equivalent and not locally homogeneous $\G$-invariant NP-structures on the space $J\times \G/\H$ for some open interval $J\subset \mathbb R$.
\end{proposition}

\subsection{Extendability over the singular orbit $\G/\K^+\cong \Sg^3$} We aim at finding necessary and sufficient conditions on the functions $f_i$ so that the $3$-form $\f$ and the corresponding metric $g_\f$ extend smoothly over one singular orbit $\G\cdot p = \G/\K^+ \cong \rm{S}^3$, where $p=\g(0)$. \par 
We first remark that when $\f$ extends smoothly over the singular orbit, then $\f_{\g(t)}(\hat e_i,\hat e_j, \hat e_k)$ is smooth in a neighborhood of the origin; therefore, as all Killing vector fields in $V^+$ vanish at $p$, we have 
$$f_0(0)=f_1(0)=f_3(0) = f_4(0) = 0.$$
Moreover the element $h = (\left(\begin{smallmatrix} 1&0\\0&-1\end{smallmatrix}\right),1)\in \K^+$ reverses the curve $\gamma$, i.e. $h(\gamma(t))=\gamma(-t)$ and its adjoint $\Ad_\G(h)=Id$, so that the functions $f_i$ extend as follows:
$$f_0\quad \mbox{{odd}};\qquad f_i\quad \mbox{{even}}\quad i=1,2,3,4.$$ 
We now consider the slice representation $\rho$ of $\K^+$ at $p$. If we write the $\ad(\gg)$-invariant decomposition $\gu = \gg + \gp$, where $\gu = \mathfrak{sp}(2) +\su(2)$ and $\gp\cong \mathbb H$, then $\rho$ can be identified with $\Ad^{\U}|_{\K^+}$ restricted to the invariant module $\gp$. If we choose standard coordinates $\{t=x_1,x_2,x_3,x_4\}$ on $\gp\cong \mathbb H$, we see that 
$$\hat e_2|_{(t,0,0,0)} = 3t\ \frac{\partial}{\partial x_2},\, \hat e_3|_{(t,0,0,0)} = 3t\ \frac{\partial}{\partial x_3},\, \hat e_4|_{(t,0,0,0)} = 3t\ \frac{\partial}{\partial x_4}. $$  
We also need an $\ad(\gk^+)$-stable complement $\gs$ in $\gg$, namely $\gs :=\{(X,0,-X)|\, X\in \su(2) \}$ and we fix the basis 
$$w_1 := (h,0,-h),\ w_2:= (e,0,-e),\ w_3:= (v,0,-v)$$
so that 
$$w_i = -\frac 13 e_{i+1} - \frac 23 e_{4+i},\quad i=1,2,3.$$
We consider the local frame along the curve $t\mapsto (t,0,0,0)\in \gp$ given by 
$$e_1,\frac{\partial}{\partial x_2}, \frac{\partial}{\partial x_3}, \frac{\partial}{\partial x_4}, 
\hat w_1|_{(t,0,0,0)}, \hat w_2|_{(t,0,0,0)}, \hat w_3|_{(t,0,0,0)}$$
with corresponding coframe $dt=x_0,dx_1,dx_2,dx_3,w^1,w^2,w^3$ satisfying  
$$e^1 = dt,\ e^2 = \frac1{3t}dx_2 - \frac 13 w^1,\ e^3 = \frac 1{3t}dx_3 - \frac 13 w^2,\ e^4 = \frac 1{3t}dx_4 - \frac 13 w^3,$$
$$e^5 = -\frac 23 w^1,\ e^6 = -\frac 23 w^2,\ e^7 = -\frac 23 w^3.$$
On the tubular neighborhood of the singular orbit $\G/\K^+$ given by $\G\times_{\K^+}\gp$, the $3$-form $\f$ defined by \eqref{phi} on the complement of the zero section is completely determined by its restrition to $\gp$. Therefore we can see $\f$ as a $\K^+$-equivariant map $\hat \f:\gp\setminus\{0\}\to \Exterior^3(\gp^* + \gs^*)$, which is fully determined by its restriction to the curve $\g(t) = (t,0,0,0)\in \gp$. We can write down the components of $\hat{\f}|_{\gamma(t)}$ ($t\neq 0$) along each of the four $\K^+$-summands in the decomposition 
$$\Exterior^3(\gp^* + \gs^*) = \Exterior^3\gp^* \oplus \Exterior^3\gs^* \oplus 
\left(\Exterior^2\gp^*\otimes\gs^*\right) \oplus \left(\gp^*\otimes \Exterior^2\gs^*\right) .$$
We will denote by $\pi_i$ the $\K^+$-equivariant projection of 
$\Exterior^3(\gp^* + \gs^*)$ onto its $i$-th summand, $i=1,\ldots,4$.\par
\noindent (a)\ 
Along $\bigwedge^3\gp^*\cong \gp$ we have 
$$\pi_1\circ \hat\f|_{\gamma(t)} = \frac 1{27t^3}f_1(t)\ dx_2\wedge dx_3\wedge dx_4$$
and therefore we may consider the $\K^+$-equivariant map $\gp\to \gp^*$ whose restriction to $\gamma$ is given by $\frac 1{27t^3}f_1(t)\ dt$ for $t\neq 0$. This extends smoothly on the whole $\gp$ if and only if $f_1(t)$ is an even smooth function of $t$ with $f_1(0)=f_1''(0)=0$ (indeed $\frac {f_1(t)}{t^4}$ must extend smoothly and we already know that $f_1$ is even and vanishes at $t=0$).\par 
\noindent (b)\ Along the trivial $\K^+$-module $\Exterior^3\gs^*$ we obtain the component 
$$\pi_2\circ \hat\f|_{\gamma(t)} = -\frac 1{27} (f_1 + 8 f_2 + 6f_3 + 12 f_4)\ w^1\wedge w^2\wedge w^3.$$
Since $w^1\wedge w^2\wedge w^3$ is $\K^+$-invariant, the extendability condition in this case boils down to the condition that $f_1 + 8 f_2 + 6f_3 + 12 f_4$ must be even. This follows from the fact each $f_i$ is even for $i=1,\ldots 4$.\par 
\noindent (c)\ Along $\gp^*\otimes \Exterior^2\gs^*$ we obtain the component 
$$\pi_3\circ \hat\f|_{\gamma(t)} = \frac 1{27t} (f_1 + 4f_3 + 4f_4)\ (dx_2\wedge w^2\wedge w^3 - dx_3\wedge w^1\wedge w^3 + dx_4\wedge w^1\wedge w^2).$$
This case can be handled in two different ways. First we note that $\K^+$ contains the normal subgroup $\mathrm{I} = \{(1,q,1)\in \K^+|\ q\in \Sp(1)\}$ which still acts transitively on the unit sphere in $\gp$, but trivially on the orbit $\G/\K^+$. Using the subgroup $\mathrm{I}$ we can determine the full expression of $\hat{\f}$ and check that its component along $\gp^*\otimes \Exterior^2\gs^*$ extends smoothly over the whole $\gp$ if and only if  $\frac 1{t^2} (f_1 + 4f_3 + 4f_4)$ extends smoothly over $t=0$. This last condition is automatic as we are supposing $f_1,f_3,f_4$ to be even functions vanishing at $t=0$. The second approach considers the $\K^+$-module $\gp^*\otimes \Exterior^2\gs^*\cong \gp^*\otimes \gs\cong Q_1\oplus Q_2$, where $Q_1\cong\gp$ and $Q_2\cong \R^8$ is the real part of the complex $\K^+$-irreducible representation $\C^2\otimes \Sg^3(\C^2)$ (here $\K^+\cong\SU(2)_1\times\SU(2)_2$ acts on $\mathbb C^2$ via $\SU(2)_1$ and on $\Sg^3(\C^2)$ via $\SU(2)_2$). Since the space $Q_2$ does not contain any non-zero fixed point vector under the action of the subgroup $\Sp(1)\cong \H\subset \K^+$, the element
$dx_2\wedge w^2\wedge w^3 - dx_3\wedge w^1\wedge w^3 + dx_4\wedge w^1\wedge w^2$ belongs to the submodule $Q_1\cong \gp$ and therefore we can consider this component of $\hat\f$ as a $\K^+$-equivariant map into $\gp$, leading to the same conclusion as above. \par 
\noindent (d)\ 
Along the summand $\Exterior^2\gp^*\otimes\gs^*$ we have 
\begin{eqnarray*} \pi_4\circ \hat\f|_{\gamma(t)} =-\frac 2{9t}f_0\cdot(dx_0\W dx_1\W w^1 + dt\W dx_2 \W w^2 + dt\W dx_3\W w^3) +  \\
-\frac{f_1+2f_3}{27 t^2}\cdot (dx_1\W dx_2\W w^3 + dx_2\W dx_3 \W w^1 - dx_1\W dx_3\W w^2).
\end{eqnarray*}
We may use the subgroup $\mathrm I$ to determine the full expression of $\pi_4\circ\hat{\f}$ on $\gp\setminus\{0\}$. 

We only write here the component of $\pi_4\circ\hat\f$ along the $3$-form $dx_0\W dx_1\W w^1$, namely 
$$-\frac 2{9t}f_0\cdot \frac{x_0^2+x_1^2}{t^2} - \frac{f_1+2f_3}{27 t^2}\cdot 
\frac{x_2^2+x_3^2}{t^2}, $$
where $t=\sum_{i=1}^4x_i^2$. This can be clearly rewritten as 

$$-\frac 2{9t}f_0 + \left(\frac{2f_0}{9t}- \frac{f_1+2f_3}{27 t^2}\right)\cdot \frac{x_2^2+x_3^2}{t^2}.$$
Therefore we see that the extendibility condition reduces to $\frac{2f_0}{9t}- \frac{f_1+2f_3}{27 t^2} = O(t^2)$. Since $f_0$ is odd and $f_1,f_3$ are even, the condition can be written as $\lim_{t\to 0} \frac{2f_0}{9t}- \frac{f_1+2f_3}{27 t^2} = 0$ or equivalently, using the fact that $f_1 = O(t^4)$ by (a), 
\begin{equation} \label{cond}
6f_0'(0) = f_3''(0).\end{equation}
It can be easily checked that the conditions on the extendability of the other components of $\pi_4\circ\hat \f$ are all equivalent to \eqref{cond}.\par 
Summing up, the $3$-form $\f$ on the regular part $\M_o$ extends on the whole tube $G\times_{\K^+}\gp$ if and only if the functions $f_i$ extend smoothly around $t=0$ with the following properties: 
\beq\label{ex1} f_0\quad \mbox{{is\ odd}},\quad f_i\quad \mbox{{are\ even,}}\quad i=1,2,3,4.\eeq
\beq\label{ex2} f_i(0)= 0,\quad i=1,3,4;\qquad \eeq
\beq\label{ex3} f_1''(0) = 0;\qquad 6f_0'(0) = f_3''(0).\eeq
When these conditions hold, the $3$-form $\f$ extends smoothy at the singular point with the expression 
\begin{eqnarray*}\f_p = A\ w^1\W w^2\W w^3 &+& B\ (dx_0\W dx_1\W w^1 + dt\W dx_2 \W w^2 + dt\W dx_3\W w^3 \\
&+& dx_1\W dx_2\W w^3 + dx_2\W dx_3 \W w^1 - dx_1\W dx_3\W w^2),\end{eqnarray*}
where $A:= -\frac{8}{27}f_2(0)$ and $B := -\frac{2}{9}f_0'(0)$. Now, it is not difficult to check that the $3$-form $\f_p$ is stable (i.e. the orbit ${\rm{GL(T_p}}\M)\cdot\f_p$ is open in $\Exterior^3\rm T_p\M^*$) if and only if $A\cdot B < 0$ and in this case the induced metric is positive definite, coinciding with the limit metric $g_\f$ at $p$. Therefore we need to consider the non-degenerancy condition 
\beq f_2(0)\cdot f_0'(0) < 0. \eeq
Actually, if we use \eqref{ex1},\eqref{ex2},\eqref{ex3} and (7)-(8) in \eqref{sys2}, we see that $ f_2(0)\cdot f_0'(0) \leq 0$, so that the only condition we need to add is 
\beq\label{ex4} f_2(0)\neq 0,\quad f_0'(0) \neq 0.\eeq
Therefore we have proved the following 
\begin{proposition} \label{extend}Let $\f$ be a $\G$-invariant $3$-form on $\M_o$  whose restriction to the curve $\gamma(t)$ ($t\neq 0$) has the expression \eqref{phi}. Assume that the form $\f$ defines an NP-structure, so that the functions $f_i's$ satisfy the algebro-differential system \eqref{sys2}. Then the form $\f$ extends smoothly to a $\G$-invariant $3$-form on $\M$ defining an NP-structure on $\M$ if and only if the functions $f_i$'s extend smoothly around $t=0$ fulfilling  the conditions \eqref{ex1}, \eqref{ex2}, \eqref{ex3}, \eqref{ex4}. \end{proposition}

\section{The Main Theorem}
In this section we will prove the existence of a one-parameter family of NP-structures on $\M$. In particular we will prove our main theorem, namely 
\begin{theorem}\label{Main} There exists a one parameter family $\mathcal F_{a}, (a\in \mathbb R^+)$ of NP-structures on $\M$. These structures are mutually non isomorphic and not locally homogeneous, with the exception of two of them, which are locally isomorphic to the structures $\mathcal P_1$ and $\mathcal P_2$ on $\Sg^7$. \end{theorem}
\begin{remark} The parameter $a\in \mathbb R^+$ measures the size of the singular orbit $\Sg^3$.  
\end{remark}

The proof of Theorem \ref{Main} will be achieved through several steps in this section.  We know how to describe $\G$-invariant NP-structures on the open dense subset $\M_o$ of $\G$-regular points, which identifies with the complement of the zero section in the bundle $\G\times_{\K^+}\mathbb H$. Given a $\G$-invariant $3$-form $\f$ on $\M_o$ which defines an NP-structure, we considered its expression \eqref{phi} along a transversal curve $\gamma$ and we could derive the algebro-differential system of equations \eqref{sys2} the functions $f_i$'s in \eqref{phi} have to satisfy. Moreover we found necessary and sufficient conditions in terms of the functions $f_i$'s so that the NP-structure on $\M_o$ extends smoothly to a global $\G_2$-structure on $M$. Now instead of solving for the functions $f_i$'s, in view of Proposition \ref{extend} we may look for smooth {\it even}  functions $h_i$ defined on some interval $(-\varepsilon,\varepsilon)$, $\varepsilon\in \mathbb R^+$, so that 
\beq\label{h}
f_0 = t\cdot h_0,\ f_1 = t^4\cdot h_1,\ f_2 =  h_2,
f_3 = t^2\cdot h_3,\ f_4 = t^2\cdot h_4,\eeq
Note that 
$$h_4 = -h_3 - \frac{\lambda}{6} h_0^2.$$
If we set $a_i:= h_i(0),\ i=0,\ldots,4$, then the extendability conditions \eqref{ex1},\eqref{ex2},\eqref{ex3}, \eqref{ex4}  are then simply written as 
\beq\label{condext} a_3 = 3a_0, \ a_0,a_2\neq 0,\eeq
We now rewrite the system \eqref{sys2} using the above defined functions $h_i$ as follows. 

\beq 
h_0' = \frac 1t (f_0' - h_0 ) = \eeq
$$ = -\frac 1t \left( h_0
+ \frac{3h_2h_3^2}{h_0^4}\right) 
- \frac{3}{2h_0^4} \left( t(h_3-h_4)(h_1h_2+h_3h_4) - 2t^3h_1h_4^2\right);$$

\beq h_1' = -\frac4t h_1 + \frac 1{t^4}f_1' = \eeq
$$= -\frac4t h_1 + \lambda\,\frac1{t^7h_0^3}\left(t^8h_1\, \frac{h_1h_2-h_3h_4}{2} - \,t^2h_3\left( t^6h_1h_4 - t^4h_3^2 \right)\right)=$$
$$= \frac 1t\left( -4h_1 + \frac{\lambda h_3^3}{ h_0^3}\right) + 
\frac {\lambda t}{2h_0^3}\left( h_1^2h_2-h_1h_3h_4 - 2\,h_1 h_3h_4  \right);$$

\beq 
h_2' = f_2' =  \frac {\lambda t}{h_0^3}\left( h_4(h_2h_3-t^2h_4^2) - \frac 12 h_2(h_1h_2-h_3h_4)\right);\eeq

\beq h_3' = -\frac 2t h_3 + \frac 1{t^2}f_3' = \eeq
$$\frac 1t\left( -2h_3 + 6h_0\right)+  \lambda\, \frac{t}{2h_0^3}\left( h_1 h_2h_3+h_3^2h_4-t^2(h_1h_4^2 + h_1h_4^2) \right)
$$
Therefore if we put $h:=(h_0,h_1,h_2,h_3)$ the system takes the form 
\beq\label{sing} h'(t) = \frac 1t A(h) + B(h,t),\qquad h(0) = \bar h =(a_0,a_1,a_2,a_3)\eeq
where $A:\mathbb R^4\to \mathbb R^4$ is given by 
\beq\label{A}A(h) = \left( -h_0- \frac{3h_2h_3^2}{h_0^4}, -4h_1 + \frac{\lambda h_3^3}{ h_0^3}, 
0, 
 -2h_3 + 6h_0\right)\eeq
 and $B:\mathbb R^4\times \mathbb R\to \mathbb R^4$ is defined by 
\begin{equation*}
 \begin{aligned} B(h,t)= (&- \frac{3}{2h_0^4} \left( t(h_3-h_4)(h_1h_2+h_3h_4) - 2t^3h_1h_4^2\right),\\
 &\frac {\lambda t}{2h_0^3}\left( h_1^2h_2-h_1h_3h_4 - 2\,h_1 h_3h_4  \right),\frac {\lambda t}{h_0^3}( h_4(h_2h_3-t^2h_4^2) - \frac 12 h_2(h_1h_2-h_3h_4)),\\
 & \frac{\lambda t}{2h_0^3}( h_1 h_2h_3+h_3^2h_4-t^2(h_1h_4^2 + h_1h_4^2))).\end{aligned}\end{equation*}
Clearly we must have $A(\bar h) = 0$, hence 
$$a_3 = 3a_0,\ a_2= -\frac 1{27}a_0^3,\ a_1 = \frac {27}4\lambda.$$
\begin{proposition}\label{propcond}
 The functions $f_i$'s as in \eqref{phi} define a $\G$-invariant NP-structure on $M = \G\times_{\K^+}\mathbb H$ (with fixed constant $\lambda$)  if and only if there exist smooth even functions $h_0,h_1,h_2,h_3$ defined on some interval $(-\varepsilon,\varepsilon)$, $(\varepsilon\in \mathbb R^+)$, satisfying the equation \eqref{sing} for $t\neq 0$ with initial condition at $t=0$ given by  $\bar h = (h_0(0),\ldots,h_3(0))$ with 
 \beq\label{init}h_0(0)=a,\ h_1(0)=\frac {27}4\l,\ h_2(0)= -\frac 1{27}a^3,\  h_3(0)= 3a,\eeq
 for some $a\in \mathbb R$, $a\neq 0$.
\end{proposition}
\begin{proof} It is enough to prove the ``if'' part. We clearly define the $f_i$'s in terms of the $h_i$'s using \eqref{h}. Then the $f_i$'s satisfy the differential system \eqref{sys2} (1)-(5) on the open set $\M_o$, while we have to prove that the algebraic conditions \eqref{sys2} (6)-(7) are also satisfied. By Lemma \ref{alg} the two quantities 
$$F_1:= f_4 + f_3 + \frac 16\,\lambda\,f_0^2,$$
$$F_2:= f_0^6 -  (f_1f_4-f_3^2)\, ( f_2f_3-f_4^2) + \frac 14 (f_1f_2-f_3f_4)^2 $$
are actually constant on the open set $t\neq 0$ and both vanish at $t=0$, hence they vanish everywhere. Moreover condition \eqref{condext} are also satisfied, so that the functions $f_i$'s define an NP-structure which extends to a $\G_2$-structure on the whole $\M$. \end{proof}
We now prove the existence in the following 
\begin{lemma} The equation \eqref{sing} admits a unique solution $h=(h_0,\ldots,h_3)$ which is smooth in an interval $(-\varepsilon,\varepsilon)$ for $(\varepsilon\in\mathbb R^+)$ and satisfies the initial conditions \eqref{init}. Moreover, the functions $h_0,\ldots,h_3$ are even.
 \end{lemma}
 \begin{proof} We use Theorem 4.7 in \cite{FH} (see also \cite{Bo}), which asserts that the singular initial value problem we are considering has a unique smooth solution provided the following conditions are fulfilled: 
 \begin{itemize}
    \item [a)] $A(\bar h) = 0$;
    \item[b)] $dA|_{\bar h}-l\cdot \rm{Id}$ is invertible for all $l\in \mathbb N$, $l\geq 1$.
   \end{itemize}
   Condition (a) has been already fixed, while we can easily compute 
   $$dA|_{\bar h} = \left[\begin{matrix} -5&0&-27 a^{-2}&2/3\\ 
-81\l a^{-1}& -4&0& 27\l a^{-1}\\ 
0&0&0&0&\\ 6&0&0&-2\end{matrix}\right]$$
whence 
$$\det(dA|_{\bar h} - l\cdot \mbox{Id}) = l(l+4)(l^2+7l+6) > 0, \quad l\geq 1.$$ 
Therefore for any $a\neq 0$, we obtain a unique solution $(h_0(t),\ldots,h_3(t))$ of the singular system \eqref{sing} with initial data $\bar h$. The fact that the solutions $h_i$ are even follows from uniqueness and the fact that $B(h,-t)=-B(h,t)$ for all $(h,t)\in\mathbb R^4\times\mathbb R$.  \end{proof}
We now investigate the question when two NP-structures determined by two solutions $h,\bar h$ are isomorphic. By the results obtained in \S 3.2, we see that the associated functions $f:=(f_i)_{i=0,\ldots,4}$ and $\bar f:= (\bar f_i)_{i=0,\ldots,4}$ are related by a transformation $\tau$ in the group generated by $\tau_{12},\tau_{13}$. A simple check shows that $\tau(f) = \bar f$ satisfies the extendability conditions \eqref{ex1},\eqref{ex2},\eqref{ex3}, if and only if $\tau = \tau_{13}$. This transformation is equivalent to reversing $a\mapsto -a$ in \eqref{init} and therefore we can restrict to $a>0$.
This concludes the proof of the main Theorem \ref{Main}.\par 
\begin{remark} The two (locally) homogeneous solutions (with $\l=1$) correspond to the values $a= 36$ (round sphere) and $a= \frac {108}5$ (squashed sphere), as it can be easily seen using \S \ref{hom} and the remark \ref{rescaling} (recall that the opposite value of $a$ gives the same NP-structure).
\end{remark}

We conclude this section pointing out some features of a possible $\G$-invariant NP-structure in the family $\mathcal F_a$ on $\M$ which extends to a global one on some $\G$-equivariant compactification $\overline \M$. Note that $\overline \M$ is diffeomorphic to either $\Sg^7$ or to $\Sg^3\times \Sg^4$.
\begin{proposition}\label{global} Any NP-structure in the family $\mathcal F_a$, $a>0$, which extends to a global NP-structure on some $\G$-equivariant compactification $\overline \M$ and has not constant curvature is proper, namely the relative cone metric has full holonomy $\Spin(7)$.\end{proposition}
\begin{proof} Let $(\overline \M,\overline g,\overline \f)$ be a $\G$-invariant NP-structure extending some element in $\mathcal F_a$, say for $a = \bar a$, and with $\overline g$ of non-constant curvature. We will denote by $\overline \G$ the connected component of the full isometry group of $(\overline \M,\overline g)$. We can also suppose that $\overline \G$ does not act transitively on $\overline \M$. Indeed, if $\overline \M \cong \Sg^7$, it is well known that there are precisely two homogeneous Einstein metrics, which correspond to the round and the squashed sphere (see e.g.~\cite{Z}). If $\overline \M \cong \Sg^3\times\Sg^4$, using a result by Kamerich (see \cite{O}, p.~274), $\overline \G$ contains a transitive subgroup $\rm N$ locally isomorphic to $\SU(2)\times \SO(5)$, acting on $\overline\M$ in a standard way. Now any $\rm N$-invariant Riemannian metric on $\overline \M$ is reducible, while $\overline g$ is irreducible. \par
Now, as $\overline M$ is compact and simply connected, our claim follows if we show that the holonomy $\mathcal H$ of the cone metric on $\overline\M \times \mathbb R^{+}$ is not $\SU(4)$ nor $\Sp(2)$.\par 
\noindent (a) We prove that  $\mathcal H \neq \SU(4)$, i.e. that $(\overline M,\overline g)$ does not carry any Sasakian structure which is not part of a $3$-Sasakian structure. If $\xi$ is the unit length Killing vector field giving the Sasakian structure, a classical theorem by Tanno (\cite{T}) states that $\xi$ belongs to the center of $\overline \gg$, since $\overline g$ has non-constant curvature. On the other hand the isotropy representation of $\K^+$ at $p$ has no non-trivial fixed vector, forcing $\xi_p=0$, a contradiction. \par 
\noindent (b)\ We now suppose that $(\overline M,\overline g)$ carries a $3$ Sasakian structure given by a Lie algebra $\gs\cong\gsp(1)$ generated by three unit length Killing vector fields $\xi_1,\xi_2,\xi_3$. Again by a result due to Tanno (\cite{T}) we know that the Lie algebra $\overline\gg$ splits as a sum of ideals 
$\overline \gg = \overline\gg_o \oplus\gs$, where $\overline\gg_o$ is the centralizer of $\gs$ in $\overline\gg$. As $\overline \G$ is supposed to act non-transitively on $\overline\M$, it has the same orbits as its subgroup $\G$. In particular, $\overline\G$ acts transitively on $\Sg^3\times \Sg^3$.
\begin{lemma} The semisimple part $\overline \gg_s$ of $\overline\gg$ is isomorphic to $3\su(2)$ or to $4\su(2)$.
 \end{lemma}
\begin{proof} The isotropy subalgebra $\gf$ of $\overline\gg_s$ at a $\G$-regular point $q$ embeds as a compact subalgebra of $\so(6)$. Looking at the list of maximal subalgebras in $\so(6)$, we see that $\dim\gf\leq 7$, unless $\gf\cong \so(5)$ or 
$\su(3),\gu(3)$. If $\gf$ contains a copy of $\su(3)$, it acts transitively on the unit sphere of $T_q(\G q)$, hence $\G q\cong \Sg^3\times\Sg^3$ has contant curvature, a contradiction; the case $\gf\cong \so(5)$ can be ruled out using a result about the gaps in the dimension of the isometry group of a Riemannian manifold (Thm. 3.3. in \cite{K}). Therefore $\dim\gf\leq 7$ and $\dim\overline\gg_s\leq 13$. On the other hand, $\overline G$ contains a subgroup isomorphic to $\SU(2)^2$ that acts on $\G q$ (almost) freely; therefore by a result in \cite{O}, Cor.3, p.~237, the algebra $\overline \gg_s$ contains an {\it ideal} $\ga$ isomorphic to $2\su(2)$. Then $\overline \gg_s = \ga\oplus \gb$ for some other semisimple ideal $\gb$. As $\gg\subseteq \overline \gg_s$ and $1\leq\dim\gb\leq 7$, we see that $\gb$ is isomorphic to $2\su(2)$ or to $\su(2)$ and our claim follows.  \end{proof}
Suppose now $\overline \gg_s=4\su(2)$. This implies that the semisimple part of $\overline \gg_o$, say $\gl$, is isomorphic to $3\su(2)$. As $[\gl,\gs]=0$, the isotropy $\gl_q$ leaves a $3$-dimensional subspace fixed, hence $\dim \gl_q\leq 3$. This implies that $\dim \rm L q =6$, where $\rm L$ is the subgroup with Lie algebra $\gl$. Then $\rm L$ has the same orbits as $\G$ and leaves each $\xi_i$ fixed, a contradiction by the same arguments used in (a). Therefore we are left with $\overline \gg_s=3\su(2)=\gg$, i.e. the ideal $\gs$ is one of the three ideals, say $\gs_1,\gs_2,\gs_3$ of $\gg$. If we denote by $g_i(t)$, $i=1,2,3$ the functions
$$g_1(t):= ||\hat e_2||^2_{\g(t)},\ g_2(t):= ||\hat e_5||^2_{\g(t)},\ g_3(t):= ||\hat e_2+\hat e_5||^2_{\g(t)},$$ 
we are reduced to considering the three possibilities when $g_i$ are constant functions. \par 
The condition $g_1(t)=f_3^2-f_1f_4={\it const}$ can be easily ruled out using Prop.\ref{propcond} and \eqref{h}. Moreover under the admissible transformation $\tau_{13}$ the function $g_2$ goes over to $g_3$, so that we can confine ourselves to the case $g_2(t) = {\it const.}$ Using Maple we can write the series expansion of the solutions $f_i$'s as well as of the 
function $g_2(t)$ for $\lambda =1$ obtaining 
$$g_2(t) = \frac 19 a^2 + (-\frac 5{576}a^2 +\frac 18 a + \frac {27}4)\ t^2 + o(t^2).$$
Therefore if $g_2$ is constant we immediately get $a = 36, -\frac{108}5$, which correspond to the known homogenous solutions. This concludes the proof.\end{proof}

\begin{remark} By the previous result, we can also show that none of the non-homogeneous Einstein metrics in the family $\mathcal F$, whenever extended to $\Sg^7$, is one of the metrics found by B\"ohm in \cite{Bo}. Indeed assume $g$ is a metric in the family $\mathcal F$ which is also in the B\"ohm's family. The NP-stucture associated to $g$ is proper, hence its full isometry group preserves the NP-structure and therefore has dimension less or equal to $9$ by Thm. 7.1 in \cite{FKMS}. On the other hand the metrics found in \cite{Bo} are invariant under a bigger group, namely $\SO(4)\times \SO(4)$.  \end{remark}


\end{document}